\documentclass[a4paper,12pt,reqno]{amsart}

\usepackage{amsmath}
\usepackage{amscd}
\usepackage{amssymb}
\usepackage{mathrsfs}
\usepackage[left=2.5cm,right=2.5cm,bottom=3cm,top=3cm]{geometry}
\newtheorem{thm}{Theorem}[section]
\newtheorem{cor}[thm]{Corollary}

\newtheorem{lem}[thm]{Lemma}
\newtheorem{prob}[thm]{Problem}
\newtheorem{prop}[thm]{Proposition}
\theoremstyle{definition}
\newtheorem{defn}[thm]{Definition}
\newtheorem{rem}[thm]{Remark}
\newtheorem{exmp}[thm]{Example}
\numberwithin{equation}{section}

\begin{document}

\title{A note on Teissier problem for nef classes}

\author{Yashan Zhang
}

\address{School of Mathematics, Hunan University, Changsha 410082, China
}
\email{yashanzh@hnu.edu.cn
}

\begin{abstract}
Teissier problem aims to characterize the equality case of Khovanskii-Teissier type inequality for $(1,1)$-classes on a compact K\"ahler manifold. When each of the involved $(1,1)$-classes is assumed to be nef and big, this problem has been solved by the previous works of Boucksom-Favre-Jonsson \cite{BFJ}, Fu-Xiao \cite{FX} and Li \cite{Li17}. In this note, we shall settle the case that the involved $(1,1)$-classes are just assumed to be nef. We also extend the results to some settings where some of the $(1,1)$-classes are not necessarily nef. By constructing examples, it is shown that our results are optimal. 
\end{abstract}

\maketitle

\section{Introduction}
\subsection{Khovanskii-Teissier inequalities and Teissier problem}
Around the year 1979, Khovanskii and Teissier independently discovered deep inequalities in algebraic geometry, which are profound analogs of Alexandrov-Fenchel inequalities in convex geometry. There are many remarkable further developments on Khovanskii-Teissier type inequalities (see e.g. \cite{BFJ,Co,De,DN,FX,Gr,La,LX,Li13,Li16,Li17,RT,Te,X} and references therein), among which we may recall the following one over an $n$-dimensional compact K\"ahler manifold $(X,\omega_X)$ as an example (see \cite{De,DN,Gr}), which will be closely related to our study here. In this note, when we write a class as $[\alpha]$, then $\alpha$ always is a smooth representative of $[\alpha]$. Arbitrarily given $[\alpha_1],...,[\alpha_{n}]\in H^{1,1}(X,\mathbb R)$ and assume that the first $n-1$ entries, $[\alpha_1],...,[\alpha_{n-1}]$, are nef, then the following Khovanskii-Teissier inequality holds:
\begin{align}\label{kt-nef}
\left(\int_X\alpha_1\wedge...\wedge\alpha_{n-2}\wedge\alpha_{n-1}\wedge\alpha_n\right)^2\ge\int_X\alpha_1\wedge...\wedge\alpha_{n-2}\wedge\alpha_{n-1}^2\cdot\int_X\alpha_1\wedge...\wedge\alpha_{n-2}\wedge\alpha_n^2.
\end{align}

Given the above inequality, it is natural to consider
\begin{prob}[Teissier \cite{Te,Te88}]\label{teissier-prob}
Characterize the equality case in \eqref{kt-nef}.
\end{prob}

In fact, a \emph{classical result} states that if further assume $[\alpha_1],...,[\alpha_{n-1}]$ are \emph{K\"ahler} classes, then the equality in \eqref{kt-nef} holds if and only if $[\alpha_{n-1}]$ and $[\alpha_n]$ are proportional in $H^{1,1}(X,\mathbb R)$ (see e.g. \cite{FX,Li16}). Moreover, if $[\alpha],[\beta]\in H^{1,1}(X,\mathbb R)$ are \emph{nef and big}, and for each $k=1,...,n-1$, the equality in \eqref{kt-nef} holds for $[\alpha_1]=...=[\alpha_k]=[\alpha]$ and $[\alpha_{k+1}]=...=[\alpha_n]=[\beta]$, then $[\alpha]$ and $[\beta]$ are proportional, thanks to the works of Boucksom-Favre-Jonsson \cite{BFJ} when $[\alpha],[\beta]$ are rational and Fu-Xiao \cite{FX} when $[\alpha],[\beta]$ are transcendental. Li \cite{Li17} further proved that if all the $n$ entries $[\alpha_1],...,[\alpha_n]$ are \emph{nef and big}, then the equality in \eqref{kt-nef} holds if and only if $[\alpha_1\wedge...\wedge\alpha_{n-2}\wedge\alpha_{n-1}]$ and $[\alpha_1\wedge...\wedge\alpha_{n-2}\wedge\alpha_{n}]$ are proportional in $H^{n-1,n-1}(X,\mathbb R)$. \\

In this note, we consider the general case of Problem \ref{teissier-prob} that all the $[\alpha_j]$'s are just assumed to be nef, dropping the bigness assumption.

\subsection{Results on Teissier Problem \ref{teissier-prob}}
We now state our results on Problem \ref{teissier-prob}.
\begin{thm}\label{thm1}
Let $X$ be an $n$-dimensional compact K\"ahler manifold, and $[\alpha_1],...,[\alpha_{n-2}],[\alpha],[\beta]\in H^{1,1}(X,\mathbb R)$. Assume either of the followings is satisfied.
\begin{itemize}
\item[(A1)] $[\alpha_1],...,[\alpha_{n-2}],[\alpha]$ are all nef and $[\beta]$ satisfies $\int_X\alpha_1\wedge...\wedge\alpha_{n-2}\wedge\beta^2\ge0$.
\item[(A2)] $[\alpha_1],...,[\alpha_{n-2}],[\alpha]$ are all nef and $\int_X\alpha_1\wedge...\wedge\alpha_{n-2}\wedge\alpha^2>0$.
\end{itemize}
Then
\begin{align}\label{kt-eq-1}
\left(\int_X\alpha_1\wedge...\wedge\alpha_{n-2}\wedge\alpha\wedge\beta\right)^2=\int_X\alpha_1\wedge...\wedge\alpha_{n-2}\wedge\alpha^2\cdot\int_X\alpha_1\wedge...\wedge\alpha_{n-2}\wedge\beta^2
\end{align}
if and only if $[\alpha_1\wedge...\wedge\alpha_{n-2}\wedge\alpha]$ and $[\alpha_1\wedge...\wedge\alpha_{n-2}\wedge\beta]$ are proportional.
\end{thm}

Theorem \ref{thm1} extends the above-mentioned results \cite{BFJ,FX,Li17} to the general nef classes, dropping the bigness assumption. 

\begin{exmp}
\begin{itemize}
\item[(1)] The (A1) is satisfied if $[\beta]$ is also nef. 
\item[(2)] The (A2) is satisfied if $[\alpha_1],...,[\alpha_{n-2}],[\alpha]$ are all nef and big, or if $[\alpha_1],...,[\alpha_{n-2}]$ are nef classes with $[\alpha_1\wedge...\alpha_{n-2}]\neq0$ and $[\alpha]$ is K\"ahler.
\end{itemize}
\end{exmp}

The following special case of Theorem \ref{thm1}(A1) may be of particular interest and applications.
\begin{thm}\label{thm1-nef}
Let $X$ be an $n$-dimensional compact K\"ahler manifold, and $[\alpha_1],...,[\alpha_{n-2}],[\alpha]$,\\$[\beta]\in H^{1,1}(X,\mathbb R)$ be nef classes. Then
\begin{align}\label{kt-eq-1+}
\left(\int_X\alpha_1\wedge...\wedge\alpha_{n-2}\wedge\alpha\wedge\beta\right)^2=\int_X\alpha_1\wedge...\wedge\alpha_{n-2}\wedge\alpha^2\cdot\int_X\alpha_1\wedge...\wedge\alpha_{n-2}\wedge\beta^2
\end{align}
if and only if $[\alpha_1\wedge...\wedge\alpha_{n-2}\wedge\alpha]$ and $[\alpha_1\wedge...\wedge\alpha_{n-2}\wedge\beta]$ are proportional.
\end{thm}

As an immediate application of Theorem \ref{thm1}(A1) (or \ref{thm1-nef}), we can slightly extend Fu-Xiao's result \cite[Theorem 2.1, (1) and (6)]{FX} on nef and big classes to the case of nef classes.
\begin{cor}\label{thm-cor}
Let $X$ be an $n$-dimensional compact K\"ahler manifold, and $[\alpha],[\beta]\in H^{1,1}(X,\mathbb R)$ are two nef classes with $[\alpha^k\wedge\beta^{n-k-1}]\neq0$ for each $k=0,1,...,n-1$. Write $s_k:=\int_X\alpha^k\wedge\beta^{n-k}$, $k=0,1,...,n$. If 
\begin{align}\label{eq-fx}
s_k^2=s_{k-1}\cdot s_{k+1}\,\,\,for\,\,\,every\,\,\,k=1,...,n-1,
\end{align}
then $[\alpha^{n-1}]$ and $[\beta^{n-1}]$ are proportional.
\end{cor}
Actually, by Theorem \ref{thm1}(A1) (or \ref{thm1-nef}), the equalities in \eqref{eq-fx} imply that $[\alpha^k\wedge\beta^{n-k-1}]$'s are proportional, $k=0,1,...,n-1$. In particular, $[\alpha^{n-1}]$ and $[\beta^{n-1}]$ are proportional, proving the desired result. This gives an alternative treatment for \cite[Theorem 2.1, (1) and (6)]{FX} (also see \cite[Corollary 3.11]{Li17}).\\

The most general setting in the inequality \eqref{kt-nef} and Problem \ref{teissier-prob} assumes that $[\alpha_1]$,$...$,\\$[\alpha_{n-2}]$,$[\alpha]$ are all nef and $[\beta]$ is arbitrary, while our Theorem \ref{thm1} characterizes its equality under additional assumptions (i.e. $\int_X\alpha_1\wedge...\wedge\alpha_{n-2}\wedge\beta^2\ge0$ in (A1) and $\int_X\alpha_1\wedge...\wedge\alpha_{n-2}\wedge\alpha^2>0$ in (A2)). Then it is natural to wonder that are these additional assumptions in Theorem \ref{thm1} necessary? Namely, is Theorem \ref{thm1} optimal? We affirmatively answer this by Example \ref{exmp-optimal} in the next section. Combining Theorem \ref{thm1} and Example \ref{exmp-optimal} settles Problem \ref{teissier-prob} completely.

On the other hand, comparing with the classical result stated below Problem  \ref{teissier-prob} and results in \cite{BFJ,FX}, it seems natural to ask: given the equality assumed in Theorem \ref{thm1} above, can we make the stronger conclusion that $[\alpha]$ and $[\beta]$ are proportional? Of course this question is interesting only when $[\alpha_1\wedge...\wedge\alpha_{n-2}]\neq0$. The Example \ref{exmp-prod} in Section \ref{sect-exmp} shows that this is not always true. Therefore, we are naturally led to consider: given nef classes $[\alpha_1],...,[\alpha_{n-2}]$, when can we conclude from the equality \eqref{kt-eq-1} that $[\alpha]$ and $[\beta]$ are proportional? This turns out to be intimately related to \emph{Hodge index theorem}, whose definition is recalled as follows.

\begin{defn}(Hodge index theorem \cite[Section 4]{DN}; also \cite[Definition 1.5]{Z})\label{defn_hi}
For any $[\Omega]\in H^{n-2,n-2}(X,\mathbb R):=H^{n-2,n-2}(X,\mathbb C)\cap H^{2n-4}(X,\mathbb R)$ and $[\eta]\in H^{1,1}(X,\mathbb R)$, we define the \emph{primitive space with respect to $([\Omega],[\eta])$} by
$$P=P_{([\Omega],[\eta])}:=\left\{[\gamma]\in H^{1,1}(X,\mathbb C)|[\Omega]\wedge[\eta]\wedge[\gamma]=0\right\}.$$
Then we say \emph{$([\Omega],[\eta])$ satisfies the Hodge index theorem} if the quadratic form
$$Q_{([\Omega],[\eta])}([\beta],[\gamma]):=\int_X\Omega\wedge\beta\wedge\overline\gamma$$
is negative definite on $P_{([\Omega],[\eta])}$. 
\end{defn}
Several known Hodge index theorems will be listed in Example \ref{example}. Here is our next theorem.
\begin{thm}\label{thm1'}
Let $X$ be an $n$-dimensional compact K\"ahler manifold. The followings are equivalent.
\begin{itemize}
\item[(1)] for any $[\alpha_1],...,[\alpha_{n-2}],[\alpha],[\beta]$ satisfying either of (A1), (A2) in Theorem \ref{thm1},
\begin{align}\label{kt-eq-1'}
\left(\int_X\alpha_1\wedge...\wedge\alpha_{n-2}\wedge\alpha\wedge\beta\right)^2=\int_X\alpha_1\wedge...\wedge\alpha_{n-2}\wedge\alpha^2\cdot\int_X\alpha_1\wedge...\wedge\alpha_{n-2}\wedge\beta^2
\end{align}
if and only if $[\alpha]$ and $[\beta]$ are proportional.
\item[(2)] $([\alpha_1\wedge...\wedge\alpha_{n-2}],[\omega_X])$ satisfies Hodge index theorem.
\item[(3)] there exists $[\theta]\in H^{1,1}(X,\mathbb R)$ such that $([\alpha_1\wedge...\wedge\alpha_{n-2}],[\theta])$ satisfies Hodge index theorem.
\end{itemize}
\end{thm}

\begin{rem}
In the algebraic setting, a special (non-mixed) version of Theorem \ref{thm1}(A2) was contained in \cite{Luo}, whose idea can be modified to proved Theorem \ref{thm1}(A2) in the K\"ahler setting (we thank Jian Xiao for pointing this out to us); also see \cite{Li17} for another argument in the case of all $n$ $(1,1)$-classes being nef and big. As will be presented in Sections \ref{sect-hi} and \ref{sect-pf}, our method in this note will take a different and slightly more general and abstract way, which enables us to unifiedly handle both items (A1) and (A2) of Theorem \ref{thm1}, as well as their extensions (see Section 4) in which some of the $(1,1)$-classes are NOT necessarily nef. Moreover, Theorem \ref{thm1} is somehow \emph{optimal}, see Example \ref{exmp-optimal}.
\end{rem}

\subsection{Organization} The remaining part is organized as follows. In Section \ref{sect-exmp} we will exhibit several examples concerning the optimality of our results. In Section \ref{sect-hi} some general properties of limits of pairs satisfying Hodge index theorem will be presented. In Section \ref{sect-pf} we will prove results stated above. In Section \ref{ext}, we discuss certain extensions of our results to more general settings where some of the $(1,1)$-classes are NOT necessarily nef.

\section{Examples}\label{sect-exmp}
We first exhibit several examples, which illustrate the optimality of our results.
\begin{exmp}[Theorem \ref{thm1} is optimal]\label{exmp-optimal}
Assume $X$ is an $n$-dimensional compact K\"ahler manifold of $d:=dim H^{1,1}(X,\mathbb R)\ge3$, and there is a holomorphic surjection $f:X\to S$ to a closed smooth Riemann surface $S$. Fix K\"ahler metrics $\omega$ on $X$ and $\chi$ on $Y$, and set $[\alpha]:=[f^*\chi]$, which is a nef class on $X$. Note that $[\omega^{n-2}\wedge\alpha]\neq0$ and 
\begin{align}\label{optimal-1}
\int_X\omega^{n-2}\wedge\alpha^2=0.
\end{align} 
Denote $$P_1:=\left\{[\gamma]\in H^{1,1}(X,\mathbb R)|[\omega^{n-1}\wedge\gamma]=0\right\}\,\, and \,\,P_2:=\left\{[\gamma]\in H^{1,1}(X,\mathbb R)|[\omega^{n-2}\wedge\alpha\wedge\gamma]=0\right\}.$$
Both $P_1$ and $P_2$ are $(d-1)$-dimensional subspace in $H^{1,1}(X,\mathbb R)$ (as $[\omega^{n-2}\wedge\alpha]\neq0$). Because $d\ge3$, by linear algebra we may fix a non-zero $[\beta]\in P_1\cap P_2$. Then $[\beta]$ satisfies

\begin{align}\label{optimal-2}
[\omega^{n-2}\wedge\alpha\wedge\beta]=0\,\,\,and\,\,\,[\omega^{n-1}\wedge\beta]=0
\end{align} 
and 
\begin{align}\label{optimal-3}
\int_X\omega^{n-2}\wedge\beta^2<0.
\end{align} 
Note that given the second equality in \eqref{optimal-2} and the fact that $[\beta]\neq0$, the \eqref{optimal-3} follows from the classical Hodge index theorem (see e.g. \cite{V}). From \eqref{optimal-1} and \eqref{optimal-2}, we obviously have the equality
$$\left(\int_X\omega^{n-2}\wedge\alpha\wedge\beta\right)^2=\int_X\omega^{n-2}\wedge\alpha^2\cdot\int_X\omega^{n-2}\wedge\beta^2.$$
However, we claim that $[\omega^{n-2}\wedge\alpha]$ and $[\omega^{n-2}\wedge\beta]$ are NOT proportional. Indeed, if $[\omega^{n-2}\wedge\beta]=c\cdot[\omega^{n-2}\wedge\alpha]$ for some $c\in\mathbb R$, then we conclude that
$$\int_X\omega^{n-2}\wedge\beta^2=c\cdot\int_X\omega^{n-2}\wedge\alpha\wedge\beta=0,$$
contradicting to \eqref{optimal-3}. This example (see \eqref{optimal-1} and \eqref{optimal-3}) shows that Theorem \ref{thm1} is optimal.
\end{exmp}

\begin{rem}
Combining Theorem \ref{thm1} and Example \ref{exmp-optimal} settles Problem \ref{teissier-prob} completely.\\
\end{rem}

The next example shows that in general the conclusion in 
Theorems \ref{thm1} (and \ref{thm1-nef}) could not be improved to $[\alpha]$ and $[\beta]$ being proportional.

\begin{exmp}($[\alpha]$ and $[\beta]$ may not be proportional)\label{exmp-prod}
Let $f:X\to Y$ be holomorphic submersion over an $m$-dimensional compact K\"ahler manifold $Y$, $1\le m\le n-2$. For $j=1,...,m$, choose $\alpha_j=f^*\chi_j$, where $\chi_j$'s are K\"ahler metrics on $Y$, and $\alpha_{m+1},...,\alpha_{n-1}$ be K\"ahler metrics on $X$ and $\alpha_n=\alpha_{n-1}+f^*\chi$ with $\chi$ a K\"ahler metric on $Y$. Then $[\alpha_1\wedge...\wedge\alpha_{n-2}]\neq0$ and the equality \eqref{kt-eq-1} holds with $[\alpha]=[\alpha_{n-1}]$ and $[\beta]=[\alpha_{n}]$, but $[\alpha_{n-1}]$ and $[\alpha_n]$ are NOT proportional. Acturally, if $[\alpha_{n-1}]$ and $[\alpha_n]$ are proportional, we can easily deduce that $[f^*\chi]$ is a K\"ahler or zero class, both of which are absurd.

Combining Theorem \ref{thm1'}, one sees that $[f^*\chi_1\wedge...\wedge f^*\chi_{m}\wedge\alpha_{m+1}\wedge...\wedge\alpha_{n-2}]$ does not satisfy Hodge index theorem (compare \cite[Theorem 1.7]{Z}).
\end{exmp}

\section{limits of pairs satisfying Hodge index theorem}\label{sect-hi}

\subsection{A generalized $m$-positivity} We first introduce the positivity condition used in the discussions.
\begin{defn}\label{m-pos}\cite[Definition 1.1]{Z}
Let $\Phi$ be a (strictly) positive $(m,m)$-form on $X$ and $\eta$ a real $(1,1)$-form on $X$. We say a real $(1,1)$-form $\alpha$ on $X$ is \emph{$(n-m)$-positive with respect to $(\Phi,\eta)$} if 
$$\Phi\wedge\eta^{n-m-k}\wedge\alpha^k>0$$	
for any $1\le k\le n-m$. In particular, the case that $(\Phi,\eta)=(\omega_X^{m},\omega_X)$ gives the original $(n-m)$-positivity with respect to a fixed K\"ahler metric $\omega_X$, and in this case we say $\alpha$ is $(n-m)$-positive with respect to $\omega_X$.
\end{defn}

\begin{defn}[generalized $m$-positivity cone]\label{defn-cone}
Fix an integer $m\le n-2$. Assume $\omega_1,...,\omega_{m}$ are K\"ahler metrics on $X$. Let 
$\Gamma\subset H^{1,1}(X,\mathbb R)$ be the convex open cone consisting of $[\phi]\in H^{1,1}(X,\mathbb R)$ that has a smooth representative $\phi$ which is $(n-m)$-positive with respect to $(\omega_1\wedge...\wedge\omega_{m},\omega_X)$, and let $\overline\Gamma$ be the closure of $\Gamma$ in $H^{1,1}(X,\mathbb R)$. Note that $\overline\Gamma$ contains the nef cone of $X$, as $\Gamma$ contains the K\"ahler cone of $X$.
\end{defn}

\subsection{Hodge index theorems}
Let $\mathscr {H}$ be the set of pair $([\Omega],[\eta])\in H^{n-2,n-2}(X,\mathbb R)\times H^{1,1}(X,\mathbb R)$ satisfying the Hodge index theorem.

\begin{exmp}[Known elements in $\mathscr {H}$]\label{example}
Several elements in $\mathscr {H}$ have been pinpointed by classical and recent works.
\begin{itemize}
\item[(1)] If $\omega$ is a K\"ahler metric on $X$, then $([\omega^{n-2}],[\omega])\in\mathscr {H}$ by the classical Hodge index theorem; if $\omega_1,...,\omega_{n-1}$ are K\"ahler metrics, then $([\omega_1\wedge...\wedge\omega_{n-2},\omega_{n-1})\in \mathscr {H}$ by \cite{C,DN,Gr}; moreover, if $\omega$ is a K\"ahler metric and $\alpha_1,...,\alpha_{n-m-1}$ are closed real $(1,1)$-forms which are $(n-m)$-positive with respect to $\omega$, then $([\omega^m\wedge\alpha_1\wedge...\wedge\alpha_{n-m-2}],[\alpha_{n-m-1}])\in\mathscr {H}$, thanks to Xiao \cite{X}.

\item[(3)] Let $m\le n-2$, $\omega_1,...,\omega_{m}$ be K\"ahler metrics and $\Gamma$ be the cone defined in Definition \ref{defn-cone}. For any $\alpha_1,...,\alpha_{n-m-1}\in\Gamma$, our previous work \cite[Theorems 1.6, 1.7]{Z} proved that $([\omega_1\wedge...\wedge\omega_{m}\wedge\alpha_1\wedge...\wedge\alpha_{n-m-2}],[\alpha_{n-m-1}])\in\mathscr {H}$.

\item[(2)] Several abstract versions of Hodge index theorem have been discovered in \cite{DN13,RT}.
\end{itemize}
\end{exmp}

\begin{rem}\label{rem_cor}\cite[Remark 2.9]{Z}
We remark some consequences of the Hodge index theorem, which should be well-known (see e.g. \cite{DN,X}). 
\begin{itemize}
\item[(1)] For any $([\Omega],[\eta])\in\mathscr{H}$, we have the Hard Lefschetz and Lefschetz Decomposition Theorems (with respect to $([\Omega],[\eta])$) as follows:
\begin{itemize}
\item[(a)] The map $[\Omega]:H^{1,1}(X,\mathbb C)\to H^{n-1,n-1}(X,\mathbb C)$ is an isomorphism;
\item[(b)] The space $H^{1,1}(X,\mathbb C)$ has a $Q$-orthogonal direct sum decomposition
$$H^{1,1}(X,\mathbb C)=P_{([\Omega],[\eta])}\oplus\mathbb C[\eta].$$
\end{itemize}

\item[(2)] For any $([\Omega],[\eta])\in\mathscr{H}$, we have the Khovanskii-Teissier type inequalities as follows. 
\begin{itemize}
\item[(c)] For any closed real $(1,1)$-forms $\phi,\psi\in H^{1,1}(X,\mathbb R)$ with $\phi$ $2$-positive with respect to $(\Omega,\eta)$, then 
$$\left(\int_X\Omega\wedge\phi\wedge\psi\right)^2\ge\left(\int_X\Omega\wedge\phi^2\right)\left(\int_X\Omega\wedge\psi^2\right)$$
with equality if and only if $[\phi]$ and $[\psi]$ are proportional.
\end{itemize}
\end{itemize}
\end{rem}

\subsection{Boundary elements in $\overline{\mathscr{H}}$}\label{subsect-bdy}
In this subsection \ref{subsect-bdy}, assume 
$$([\Omega],[\eta])=\lim([\Omega_i],[\eta_i])\,\,\, in \,\,\,H^{n-2,n-2}(X,\mathbb R)\times H^{1,1}(X,\mathbb R),$$ 
where $([\Omega_i],[\eta_i])\in\mathscr{H}$; in other words, $([\Omega],[\eta])\in\overline{\mathscr{H}}$, the closure of $\mathscr{H}$ in $H^{n-2,n-2}(X,\mathbb R)\times H^{1,1}(X,\mathbb R)$. 

We shall investigate the properties of such $([\Omega],[\eta])$, which slightly extends \cite[Section 4]{DN} and will be used in our later discussions.

\begin{exmp}[Known elements on the boundary of $\overline{\mathscr{H}}$]\label{exmp-hi-bdy}
\noindent\par(1) (\cite[Section 4]{DN}) Let $[\alpha_1],...,[\alpha_{n-2}]$ be nef $(1,1)$-classes with $[\alpha_1\wedge...\wedge\alpha_{n-2}]\neq0$, and $\omega$ a K\"ahler metric. Note that both $[\alpha_1\wedge...\wedge\alpha_{n-2}]\wedge[\omega]$ and $[\alpha_1\wedge...\wedge\alpha_{n-2}]\wedge[\omega^2]$ can be represented by non-zero positive currents. Then 
$$([\alpha_1\wedge...\wedge\alpha_{n-2}],[\omega])\in\overline{\mathscr{H}}\,\,\,and \,\,\,[\alpha_1\wedge...\wedge\alpha_{n-2}]\wedge[\omega^2]\neq0.$$\\
\noindent(2)Fix an integer $m\le n-2$. Assume $\omega_1,...,\omega_{m}$ are K\"ahler metrics on $X$. Let 
$\Gamma\subset H^{1,1}(X,\mathbb R)$ be the cone defined in Definition \ref{defn-cone}. Suppose $[\alpha_1],...,[\alpha_{n-m-2}]\in\overline\Gamma$ and set $[\Omega]:=[\omega_1\wedge...\wedge\omega_{m}\wedge\alpha_{1}\wedge...\wedge\alpha_{n-m-2}]$. Assume $[\Omega]\neq0$. Then for any K\"ahler metric $\omega$, by Example \ref{example}(2) we know $([\Omega],[\omega])\in \overline{\mathscr{H}}$, as each $[\alpha_j]$ is a limit of elements in $\Gamma$. Moreover, for any $n-m-2$ closed real $(1,1)$-forms $\tilde\alpha_1,...,\tilde\alpha_{n-m-2}$ with $[\tilde\alpha_j]\in\Gamma$, by G{\aa}rding theory for hyperbolic polynomials (see e.g. \cite[Theorem 2.2(1)]{Z}), locally we have, for any $i,j$ and any K\"ahler metric $\tilde\omega$,
$$\omega_1\wedge...\wedge\omega_{m}\wedge\tilde\alpha_{1}\wedge...\wedge\tilde\alpha_{n-m-2}\wedge\sqrt{-1}dz^i\wedge d\bar z^i\wedge\sqrt{-1}dz^j\wedge d\bar z^j\ge0$$
and 
$$\omega_1\wedge...\wedge\omega_{m}\wedge\tilde\alpha_{1}\wedge...\wedge\tilde\alpha_{n-m-2}\wedge\tilde\omega\wedge\sqrt{-1}dz^j\wedge d\bar z^j>0$$
from which, by limiting $\tilde\alpha_j$'s, one sees that both $[\Omega]\wedge[\omega]$ and $[\Omega]\wedge[\omega^2]$ can be represented by non-zero positive currents. Therefore, $[\Omega]\wedge[\omega^2]\neq0$.\\

\noindent(3) Assume $[\alpha_1],...,[\alpha_{n-2}],[\eta]$ are nef and big, and set $[\Omega]=[\alpha_1\wedge...\wedge\alpha_{n-2}]$. Then obviously $([\Omega],[\eta])\in\overline{\mathscr{H}}$ and $[\Omega]\wedge[\omega^2]\neq0$.\\

\noindent(4) We may particularly mention that, given the abstract versions of Hodge index theorem (see \cite{DN13,RT}), we have some abstract elements on the boundary of $\overline{\mathscr{H}}$.

\end{exmp}

In general, $([\Omega],[\eta])$ may no longer satisfy the Hodge index  theorem; however, it still has some good properties. 

\begin{lem}[Lefschetz decomposition]\label{lem-ld}
Suppose $([\Omega],[\eta])\in\overline{\mathscr{H}}$ and $[\Omega]\wedge[\eta^2]\neq0$. Then any $[\beta]\in H^{1,1}(X,\mathbb C)$ can be uniquely decomposed as 
$$[\beta]=c\cdot[\eta]+[\gamma], \,\,\,with\,\,c\in\mathbb C\,\,and\,\,[\gamma]\in P_{([\Omega],[\eta])}.$$
\end{lem}
\begin{proof}
By definition, $$([\Omega],[\eta])=\lim([\Omega_i],[\eta_i])\,\,\, in \,\,\,H^{n-2,n-2}(X,\mathbb R)\times H^{1,1}(X,\mathbb R)$$ 
for a family $([\Omega_i],[\eta_i])\in\mathscr{H}$. Then by Remark \ref{rem_cor} we know $[\beta]=c_i[\eta_i]+[\gamma_i]$ for some $c_i\in\mathbb C$ and $[\gamma_i]\in P_{([\Omega_i],[\eta_i])}$. Therefore,
$$[\Omega_i]\wedge[\eta_i]\wedge[\beta]=c_i[\Omega_i]\wedge[\eta_i^2].$$
Letting $i\to\infty$ and combining $[\Omega]\wedge[\eta^2]\neq0$, we see that $c_i$'s are uniformly bounded, which in turn implies that $[\gamma_i]$'s are uniformly bounded. Up to passing to a subsequence, we may assume that $c_i\to c$ and $[\gamma_i]\to[\gamma]$. Therefore, 
$$[\beta]=c\cdot[\eta]+[\gamma].$$
Note that $[\gamma]\in P_{([\Omega],[\eta])}$, since $[\Omega]\wedge[\eta]\wedge[\gamma]=\lim[\Omega_i]\wedge[\eta_i]\wedge[\gamma_i]=0$. We have proved the existence of the required decomposition.

To see the uniqueness, assume there exists another decomposition $[\beta]=c'\cdot[\eta]+[\gamma']$ for some $c'\in\mathbb C$ and $[\gamma']\in P_{([\Omega],[\eta])}$, which gives that $(c-c')[\eta]=[\gamma']-[\gamma]\in P_{([\Omega],[\eta])}$, i.e. $(c-c')[\Omega]\wedge[\eta]\wedge[\eta]=0$. So $c=c'$ and $[\gamma]=[\gamma']$.

The proof is completed.
\end{proof}

\begin{lem}\label{lem-hi}
Suppose $([\Omega],[\eta])\in\overline{\mathscr{H}}$ and $[\Omega]\wedge[\eta]\neq0$. Then the quadratic form
$$Q_{([\Omega],[\eta])}([\beta],[\gamma]):=\int_X\Omega\wedge\beta\wedge\overline\gamma$$
is semi-negative definite on $P_{([\Omega],[\eta])}$. 
\end{lem}
Note that here we only assume $[\Omega]\wedge[\eta]\neq0$.
\begin{proof}
Write $([\Omega],[\eta])=\lim([\Omega_i],[\eta_i])$ as above. Since $Q_{([\Omega_i],[\eta_i])}$ is negative definite on $P_{([\Omega_i],[\eta_i])}$, by continuity (here we need $[\Omega]\wedge[\eta]\neq0$) we immediately conclude that $Q_{([\Omega],[\eta])}$
is semi-negative definite on $P_{([\Omega],[\eta])}$.
\end{proof}

\begin{lem}[characterization of zero-eigenvector]\label{lem-0eigenvector}
Suppose $([\Omega],[\eta])\in\overline{\mathscr{H}}$ and $[\Omega]\wedge[\eta^2]\neq0$. Then for $[\gamma]\in P_{([\Omega],[\eta])}$, 
$$Q_{([\Omega],[\eta])}([\gamma],[\gamma])=0 \Longleftrightarrow [\Omega]\wedge[\gamma]=0.$$
\end{lem}
\begin{proof}
It suffices to prove the ``$\Rightarrow$'' direction, as the other direction is trivial. The proof is similar to \cite[Proposition 4.1]{DN}. Since we are in a slightly more abstract setting, for convenience let's present some details. Given $[\gamma]\in P_{([\Omega],[\eta])}$ with $Q_{([\Omega],[\eta])}([\gamma],[\gamma])=0$, to check that $[\Omega]\wedge[\gamma]=0$, by Poincar\'e duality and the definition of $Q=Q_{[\Omega],[\eta])}$, it suffices to examine that $Q([\gamma],\cdot)$ is the zero functional on $H^{1,1}(X,\mathbb C)$. On the other hand, note that $H^{1,1}(X,\mathbb C)=\mathbb C[\eta]\oplus P_{([\Omega],[\eta])}$ by Lemma \ref{lem-ld}, therefore, to check $Q([\gamma],\cdot)$ is the zero functional on $H^{1,1}(X,\mathbb C)$, it suffices to examine that $Q([\gamma],[\eta])=0$ and $Q([\gamma],\cdot)=0$ on $P_{([\Omega],[\eta])}$. The former one is trivially true as $[\gamma]\in P_{([\Omega],[\eta])}$. To see the latter one, arbitrarily take $[\gamma']\in P_{([\Omega],[\eta])}$. By Lemma \ref{lem-hi} and $Q_{([\Omega],[\eta])}([\gamma],[\gamma])=0$, we have, for any $t\in\mathbb R$,
\begin{align}
0\ge Q([\gamma]+t[\gamma'],[\gamma]+t[\gamma'])=2t\cdot Re(Q([\gamma],[\gamma']))+t^2\cdot Q([\gamma'],[\gamma']),
\end{align}
from which, by an elementary analysis, we conclude that $Re(Q([\gamma],[\gamma']))=0$. Similarly consider $Q(\sqrt{-1}[\gamma]+t[\gamma'],\sqrt{-1}[\gamma]+t[\gamma'])$, we conclude that $Im(Q([\gamma],[\gamma']))=0$. Therefore, $Q([\gamma],[\gamma'])=0$, as desired.

The proof is completed.
\end{proof}

The above arguments can be applied to prove
\begin{lem}\label{lem-int}
Suppose $([\Omega],[\eta])\in\overline{\mathscr{H}}$ and $[\Omega]\wedge[\eta^2]\neq0$. Then 
$$\int_X[\Omega]\wedge[\eta^2]\neq0$$.
\end{lem}
\begin{proof}
Indeed, since $H^{n,n}(X,\mathbb R)$ is one-dimensional, $[\Omega]\wedge[\eta^2]=c\cdot[\omega_X^n]$ for some nonzero $c\in\mathbb R$. So $\int_X[\Omega]\wedge[\eta^2]\neq0$.

Alternatively we may apply the above arguments to check this. Assume a contradiction that $\int_X[\Omega]\wedge[\eta^2]=0$, i.e. $Q([\eta],[\eta])=0$. Obviously we also have $Q([\eta],\cdot)=0$ on $P_{([\Omega],[\eta])}$. Therefore, $Q([\eta],\cdot)=0$ on $H^{1,1}(X,\mathbb C)$, and hence $[\Omega]\wedge[\eta]=0$, which contradicts to $[\Omega]\wedge[\eta^2]\neq0$.

The proof is completed.
\end{proof}

\begin{lem}\label{lem-ld'}
Suppose $([\Omega],[\eta])\in\overline{\mathscr{H}}$ and $[\Omega]\wedge[\eta^2]\neq0$. Then for $[\gamma]\in H^{1,1}(X,\mathbb C)$, 
$$\int_X[\Omega]\wedge[\eta]\wedge[\gamma]=0 \Longleftrightarrow [\gamma]\in P_{([\Omega],[\eta])}.$$
\end{lem}
\begin{proof}
It suffices to prove the ``$\Rightarrow$'' direction, as the other direction is trivial. The proof is a simple application of Lemma \ref{lem-ld}. Indeed, by Lemma \ref{lem-ld} we write $[\gamma]=c\cdot[\eta]+[\gamma']$ for some $c\in\mathbb C$ and $[\gamma']\in P_{([\Omega],[\eta])}$, then
$$0=\int_X[\Omega]\wedge[\eta]\wedge[\gamma]=c\cdot\int_X[\Omega]\wedge[\eta^2],$$
which, combining Lemma \ref{lem-int}, gives $c=0$.

The proof is completed.
\end{proof}

\section{Proofs of the results}\label{sect-pf}
Given the above preparations, now we are ready to prove our results stated in the introduction. 

\subsection{A general result} 
We first prove a general result for \emph{abstract} elements on the boundary of $\overline{\mathscr{H}}$.
\begin{prop}\label{thm4}
Suppose $([\Omega],[\eta])\in\overline{\mathscr{H}}$ and $[\Omega]\wedge[\eta^2]\neq0$. Then for any $[\beta]\in H^{1,1}(X,\mathbb R)$,
\begin{align}\label{pf'-kt}
\left(\int_X[\Omega]\wedge[\eta]\wedge[\beta]\right)^2=\int_X[\Omega]\wedge[\eta^2]\int_X[\Omega]\wedge[\beta^2].
\end{align} 
if and only if $[\Omega\wedge\eta]$ and $[\Omega\wedge\beta]$ are proportional.
\end{prop}

\begin{proof}
Denote $Q=Q_{([\Omega],[\alpha])}$. Consider a function $f(t)$ for $t\in\mathbb R$:
$$f(t):=Q([t\eta+\beta],[t\eta+\beta])=t^2Q([\eta],[\eta])+2tQ([\eta],[\beta])+Q([\beta],[\beta]).$$
May assume $Q([\eta],[\eta])>0$. The equality \eqref{pf'-kt} implies that there exists exactly one $t_0\in\mathbb R$ with $f(t_0)=0$, and hence $f(t)>0$ for any $t\neq t_0$. We seperate the discussions into the following two cases.\\

\noindent\underline{\textbf{Case 1}}: $Q([\beta],[\beta])=0$. 

Then by an elementary analysis we conclude that $Q([\eta],[\beta])=0$, which implies $[\beta]\in P_{([\Omega],[\eta])}$ by Lemma \ref{lem-ld'}. Therefore, by Lemma \ref{lem-0eigenvector} we see $[\Omega]\wedge[\beta]=0$, and hence $[\Omega]\wedge[\eta]$ and $[\Omega]\wedge[\beta]$ are trivially proportional.\\

\noindent\underline{\textbf{Case 2}}: $Q([\beta],[\beta])\neq0$. 

Then $Q([\beta],[\beta])>0$. Up to rescaling $[\beta]$, we may assume $Q([\beta],[\beta])=Q([\eta],[\eta])$. Moreover, up to changing $[\beta]$ to $-[\beta]$, by \eqref{pf'-kt} we may assume $Q([\eta],[\beta])=Q([\beta],[\beta])=Q([\eta],[\eta])$. Therefore, the unique zero of $f$ is $t_0=-1$, and we arrive at
\begin{align}\label{eq-0}
Q([\eta-\beta],[\eta-\beta])=0,
\end{align}
and
$$\int_X\Omega\wedge\eta\wedge(\eta-\beta)=0.$$
By Lemma \ref{lem-ld'} the latter implies 
\begin{align}\label{eq-00}
[\eta-\beta]\in P_{([\Omega],[\eta])}.
\end{align}
Combining \eqref{eq-0} and \eqref{eq-00}, by Lemma \ref{lem-0eigenvector} we conclude that $[\Omega]\wedge[\eta-\beta]=0$, i.e. $[\Omega]\wedge[\eta]$ and $[\Omega]\wedge[\beta]$ are proportional.

Proposition \ref{thm4} is proved.
\end{proof}

\subsection{Proof of Theorem \ref{thm1}} We now prove Theorem \ref{thm1}.

\begin{proof}[Proof of Theorem \ref{thm1}]
Given the setup of Theorem \ref{thm1}, set $[\Omega]:=[\alpha_1]\wedge...\wedge[\alpha_{n-2}]$. May assume $[\Omega]\neq0$.

Note that Item (A2) is already contained in Proposition \ref{thm4} as a special case.

Next we look at Item (A1). Firstly we have $([\Omega],[\omega_X])\in\overline{\mathscr{H}}$ and $[\Omega]\wedge[\omega_X^2]\neq0$ (see Examples \ref{exmp-hi-bdy}(1)). Denote $Q:=Q_{([\Omega],[\omega_X])}$ If $\int_X\Omega\wedge\alpha^2\neq0$, applying Proposition \ref{thm4} gives the desired result. Now assume $\int_X\Omega\wedge\alpha^2=0$. Also assume $[\Omega\wedge\alpha]\neq0$ (otherswise the required result is trivially true). The equality \eqref{kt-eq-1} implies $\int_X\Omega\wedge\alpha\wedge\beta=0$, therefore, $[\beta]\in P_{([\Omega],[\alpha])}$ and by Lemma \ref{lem-hi},
\begin{align}\label{betale0}
\int_X\Omega\wedge\beta^2\le0.
\end{align}
On the other hand, by assumption, $\int_X\Omega\wedge\beta^2\ge0$. So we conclude that 
\begin{align}\label{beta=0}
\int_X\Omega\wedge\beta^2=0.
\end{align}
In this case, we have
\begin{align}\label{eq-a1}
Q([s\alpha+t\beta],[s\alpha+t\beta])=0,\,\,\,\forall s,t\in\mathbb R.
\end{align}
Moreover, we can always fix $s_0,t_0\in\mathbb R$, at least one of which is non-zero, such that
$$\int_X\Omega\wedge\omega_X\wedge(s_0\alpha+t_0\beta)=0,$$
which implies 
\begin{align}\label{eq-a1'}[s_0\alpha+t_0\beta]\in P_{([\Omega],[\omega_X])}
\end{align} 
by Lemma \ref{lem-ld'}. Applying Lemma \ref{lem-0eigenvector} with \eqref{eq-a1} and \eqref{eq-a1'}, we conclude $[\Omega]\wedge[s_0\alpha+t_0\beta]=0$, proving the desired result.

Theorem \ref{thm1} is proved.
\end{proof}

\begin{rem}\label{pf-rem}
In Li's arguments in \cite[Proof of Theorem 3.9]{Li17} proving the case that all the involved classes $[\alpha_1],...,[\alpha_{n-2}],[\alpha],[\beta]$ are nef and big, Shephard inequality (see \cite{Sh} or \cite[Theorem 3.5]{Li17}) plays a crucial role (note that Shephard inequality also works when all the involved classes are nef). However, for the setting of Theorem \ref{thm1}, the last entry $[\beta] $ is no longer assumed to be nef, resulting that Shephard inequality can not be applied. Our proof here provides an alternative proof for  \cite[Theorem 3.9]{Li17} without using Shephard inequality.
\end{rem}

\subsection{Proof of Theorem \ref{thm1'}}
\begin{proof}[Proof of Theorem \ref{thm1'}]

Note that the implication $(2)\Rightarrow(3)$ is trivial. 

\underline{$(3)\Rightarrow(1)$}: Since $([\Omega],[\theta])\in\mathscr{H}$, the map $[\Omega]\wedge\cdot: H^{1,1}(X,\mathbb C)\to H^{n-1,n-1}(X,\mathbb C)$ is an isomorphism (see Remark \ref{rem_cor}(a)), and hence $[\Omega]\wedge[\alpha]$ and $[\Omega]\wedge[\beta]$ are proportional if and only if $[\alpha]$ and $[\beta]$ are proportional. Then applying Theorem \ref{thm1} gives the result.

\underline{$(1)\Rightarrow(2)$}: Let's only look at the Case (A1), as the other cases can be checked similarly. Denote $[\Omega]:=[\alpha_1\wedge...\wedge\alpha_{n-2}]$. Note that, since condition (1) holds, $[\Omega]\neq0$ and hence $\int_X[\Omega]\wedge[\omega^2]\neq0$ for any K\"ahler metric $\omega$ (see Example \ref{exmp-hi-bdy}(1)).  Assume a contradiction that $([\Omega],[\omega_X])$ does not satisfy Hodge index theorem. Then $([\Omega],[\omega_X])\in\overline{\mathscr{H}}\setminus\mathscr{H}$. By definition of $\mathscr{H}$ and Lemme \ref{lem-hi}, there exists $[\gamma]\in P_{([\Omega],[\omega_X])}\setminus\{0\}$ such that 
\begin{align}\label{gamma}
Q_{([\Omega],[\omega_X])}([\gamma],[\gamma])=0,
\end{align} 
and hence 
\begin{align}\label{gamma'}
[\Omega]\wedge[\gamma]=0
\end{align}
by Lemma \ref{lem-0eigenvector}. May assume $[\gamma]$ is real, and then fix a K\"ahler metric $\omega$ such that $[\omega+\gamma]$ is a K\"ahler class. Then the equality \eqref{kt-eq-1'} holds with $[\alpha]=[\omega]$ and $[\beta]=[\omega+\gamma]$ by \eqref{gamma'}. So $[\omega]$ and $[\omega+\gamma]$ are proportional, which in turn gives that $[\omega]$ and $[\gamma]$ are proportional, as both $[\omega]$ and $[\gamma]$ are non-zero. Write $[\gamma]=c\cdot[\omega]$ with $c\neq0$. We arrive at
$$Q_{([\Omega],[\omega_X])}([\gamma],[\gamma])=c^2\cdot\int_X[\Omega]\wedge[\omega^2]\neq0,$$
contradicting to \eqref{gamma}. The proof of implication $(1)\Rightarrow(2)$ is completed.

Theorem \ref{thm1'} is proved.
\end{proof}

\section{Extensions}\label{ext}
In this section, we fix an integer $m\le n-2$ and $m$ K\"ahler metrics $\omega_1,...,\omega_{m}$ on $X$. Let 
$\Gamma\subset H^{1,1}(X,\mathbb R)$ be the convex open cone defined in Definition \ref{defn-cone}, and let $\overline\Gamma$ be the closure of $\Gamma$ in $H^{1,1}(X,\mathbb R)$. Arbitrarily take $[\alpha_1],...,[\alpha_{n-m-2}],[\alpha]\in\overline\Gamma$, $[\beta]\in H^{1,1}(X,\mathbb R)$ and set $[\Omega]:=[\omega_1]\wedge...\wedge[\omega_{m}]\wedge[\alpha_{1}]\wedge...\wedge[\alpha_{n-m-2}]$, then we have
\begin{align}\label{kt-general}
\left(\int_X\Omega\wedge\alpha\wedge\beta\right)^2\ge\int_X
\Omega\wedge\alpha^2\cdot\int_X\Omega\wedge\beta^2.
\end{align}
The \eqref{kt-general} is contained in \cite[Theorem 1.6 and Theorem 2.10]{Z} (also see \cite{Co,X} for the special case that $\omega_1=...=\omega_{m}=\omega_X$) when each $[\alpha_j]\in\Gamma$; then taking a limit gives \eqref{kt-general}. Comparing with \eqref{kt-nef}, $[\alpha_j]$'s and $[\alpha]$ in \eqref{kt-general} are not necessarily nef, however we pay the price that the $\omega_j$'s in \eqref{kt-general} have been assumed to be K\"ahler (see Example \ref{exmp-semipos} for a special case where none of the involved $(1,1)$-classes is Kahler). Given the inequality \eqref{kt-general}, it is also natural to consider the following generalized Teissier problem:
\begin{prob}[Teissier, a generalized version]\label{teissier-prob'}
Characterize the equality case in \eqref{kt-general}.
\end{prob}

\subsection{Results on Problem \ref{teissier-prob'}}
The above discussions in Sections \ref{sect-hi} and \ref{sect-pf} are valid for \emph{abstract} elements on the boundary of $\overline{\mathscr{H}}$, and hence can be identically applied to study Problem \ref{teissier-prob'} and give the following results.

\begin{thm}\label{thm2}
Suppose $[\alpha_1],...,[\alpha_{n-m-2}],[\alpha],[\beta]\in H^{1,1}(X,\mathbb R)$ and set $[\Omega]:=[\omega_1\wedge...\wedge\omega_{m}\wedge\alpha_{1}\wedge...\wedge\alpha_{n-m-2}]$. Assume either of the followings is satisfied.
\begin{itemize}
\item[(C1)] $[\alpha_1],...,[\alpha_{n-m-2}],[\alpha]\in\overline\Gamma$ and $\int_X\Omega\wedge\beta^2\ge0$.
\item[(C2)] $[\alpha_1],...,[\alpha_{n-m-2}],[\alpha]\in\overline\Gamma$, and $\int_X\Omega\wedge\alpha^2>0$.
\end{itemize}
Then 
\begin{align}
\left(\int_X\Omega\wedge\alpha\wedge\beta\right)^2=\int_X\Omega
\wedge\alpha^2\cdot\int_X\Omega\wedge\beta^2\nonumber
\end{align}
if and only if $[\Omega\wedge\alpha]$ and $[\Omega\wedge\beta]$ are proportional.
\end{thm}

\begin{rem}
Combining Theorem \ref{thm2} and Example \ref{exmp-optimal} settles Problem \ref{teissier-prob'} completely.
\end{rem}

\begin{rem}
The analogs of Theorem \ref{thm1-nef} and Corollary \ref{thm-cor} can be carried out similarly; here we just mention the latter one. Assume $[\alpha],[\beta]\in\overline{\Gamma}$ with $[\omega_1\wedge...\wedge\omega_m\wedge\alpha^{k}\wedge\beta^{n-m-k-1}]\neq0$ for $k=0,1,...,n-m-1$. Write $s_k:=\int_X\omega_1\wedge...\wedge\omega_m\wedge\alpha^{k}\wedge\beta^{n-m-k}$, $k=0,1,...,n-m$. If
$$s_k^2=s_{k-1}\cdot s_{k+1},\,\,k=1,...,n-m-1,$$
then $[\omega_1\wedge...\wedge\omega_m\wedge\alpha^{n-m-1}]$ and $[\omega_1\wedge...\wedge\omega_m\wedge\beta^{n-m-1}]$ are proportional.
\end{rem}

We also have the analog of Theorem \ref{thm2} as follows.
\begin{thm}\label{thm2'}
The followings are equivalent.
\begin{itemize}
\item[(1)] for any $[\alpha_1],...,[\alpha_{n-2}],[\alpha],[\beta]$ (and then set $[\Omega]:=[\omega_1\wedge...\wedge\omega_{m}\wedge\alpha_{1}\wedge...\wedge\alpha_{n-m-2}]$) satisfying either of (C1), (C2) in Theorem \ref{thm2},
\begin{align}\label{kt-eq-1+}
\left(\int_X\Omega\wedge\alpha\wedge\beta\right)^2=
\int_X\Omega\wedge\alpha^2\cdot\int_X\Omega\wedge\beta^2
\end{align}
if and only if $[\alpha]$ and $[\beta]$ are proportional.
\item[(2)] $([\Omega],[\omega_X])$ satisfies Hodge index theorem.
\item[(3)] there exists $[\theta]\in H^{1,1}(X,\mathbb R)$ such that $([\Omega],[\theta])$ satisfies Hodge index theorem.
\end{itemize}
\end{thm}

We end this note by considering a special case where none of the involved $(1,1)$-classes is K\"ahler and some of the involved $(1,1)$-classes are not necessarily nef.
\begin{exmp}\label{exmp-semipos}
Let $f:X^n\to Y^m$ be a holomorphic submersion, $m<n$, and fix $2<k<m$. Assume $[\eta_1],...,[\eta_{n-m}]$ are nef $(1,1)$-classes on $X$, $\chi_1,...,\chi_{m-k},\chi_Y$ K\"ahler metrics on $Y$ and $\alpha_1,...,\alpha_{k-2},\alpha_{Y}$ closed real $(1,1)$-forms on $Y$ which are $k$-positive with respect to $(\chi_1\wedge...\wedge\chi_{m-k},\chi_Y)$. Set $\Omega_0:=\chi_1\wedge...\wedge \chi_{m-k}\wedge\alpha_1\wedge...\wedge\alpha_{k-2}$, and
$$\Gamma_0:=\{[\alpha]\in H^{1,1}(Y,\mathbb R)|[\alpha]\,\,is\,\,2-positive \,\,w.r.t.\,\,(\Omega_0,\alpha_Y)\}.$$
For any $\epsilon>0$ and $[\alpha]\in\Gamma_0$, by \cite[Theorem 1.7]{Z} we know $([\eta_{1,\epsilon}\wedge...\wedge\eta_{n-m,\epsilon}\wedge f^*\Omega_0],[f^*\alpha])$ satisfies Hodge index theorem, here $[\eta_{i,\epsilon}]:=[\eta_i]+\epsilon[\omega_X]$. Now we set 
$$\Omega:=\eta_1\wedge...\wedge\eta_{n-m}\wedge f^*\Omega_0=\eta_1\wedge...\wedge\eta_{n-m}\wedge f^*\chi_1\wedge...\wedge f^*\chi_{m-k}\wedge f^*\alpha_1\wedge...\wedge f^*\alpha_{k-2};$$
note that in the definition of $[\Omega]$, $[\eta_i]$'s and $[f^*\chi_i]$'s are only nef (not necessarily K\"ahler) and $[f^*\alpha_i]$'s are not necessarily nef.
By an obvious limiting procedure we have the inequality:
\begin{equation}
\left(\int_X\Omega\wedge f^*\alpha\wedge\beta\right)^2\ge \int_X\Omega
\wedge f^*\alpha^2\cdot\int_X\Omega\wedge\beta^2
\end{equation}
for any $[\alpha]\in\bar\Gamma_0$ and $\beta\in H^{1,1}(X,\mathbb R)$; furthermore, if either $\int_X\Omega\wedge\beta^2\ge0$ or $\int_X\Omega
\wedge f^*\alpha^2>0$, then the equality holds if and only if
$[\Omega
\wedge f^*\alpha]$ and $[\Omega
\wedge \beta]$ are proportional.

\end{exmp}

\section*{Acknowledgements}
The author is grateful to Jian Xiao for a number of valuable comments on the results.

\end{document}